\documentclass[10pt]{amsart}

 \usepackage{hyperref}

\hypersetup{colorlinks=true, urlcolor=blue, citecolor=blue, linkcolor=blue}

\usepackage{url}

\usepackage{amsmath,amsfonts, latexsym,graphicx, footmisc, amssymb}

\newcommand{\U}{{\mathcal U}}

\newcommand{\C}{{\mathbb C}}
\newcommand{\Z}{{\mathbb Z}}

\newcommand{\N}{{\mathbb N}}
\newcommand{\cL}{{\mathbb L}}

\newcommand{\W}{{\mathcal W}}

\newcommand{\strat}{{\mathfrak S}}
\newcommand{\Proj}{{\mathbb P}}
\newcommand{\hyp}{{\mathbb H}}

\newcommand{\arrow}[1]{\stackrel{#1}{\longrightarrow}}
\newcommand{\Adot}{\mathbf A^\bullet}

\newcommand{\Pdot}{\mathbf P^\bullet}

\newcommand{\p}{{\mathbf p}}

\newcommand{\ms}{{\operatorname{SS}}}

\newcommand{\relcon}{{\overline{T^*_f\U}}}

\newtheorem{defn0}{Definition}[section]
\newtheorem{prop0}[defn0]{Proposition}
\newtheorem{conj0}[defn0]{Conjecture}
\newtheorem{thm0}[defn0]{Theorem}
\newtheorem{lem0}[defn0]{Lemma}
\newtheorem{corollary0}[defn0]{Corollary}
\newtheorem{example0}[defn0]{Example}
\newtheorem{remark0}[defn0]{Remark}
\newtheorem{question0}[defn0]{Question}

\newenvironment{prop}{\begin{prop0}}{\end{prop0}}

\newenvironment{thm}{\begin{thm0}}{\end{thm0}}

\newcommand{\propref}[1]{Proposition~\ref{#1}}
\newcommand{\thmref}[1]{Theorem~\ref{#1}}

\title[A Result on Relative Conormal Spaces]{A Result on Relative Conormals Spaces}

\subjclass[2010]{32B15, 32C18, 32B10, 32S25, 32S15, 32S55}
\keywords{relative conormal space, nearby cycles, vanishing cycles}

\author{David B. Massey}

\date{}

\begin{document}

\begin{abstract} We prove a result on the relationship between the relative conormal space of an analytic function $f$ on affine space and the relative conormal space of $f$ restricted to a hyperplane slice, at a point where the relative conormal space of $f$ is ``microlocally trivial''.
\end{abstract}

\maketitle




\section{Introduction} 

Suppose that $\U$ is a connected, open subset of $\C^{n+1}$; we use $(z_0, \dots, z_n)$ as coordinates on $\U$. Let $p\in \U$, and let $f:\U\rightarrow\C$ be a non-constant complex analytic function. We will assume, without loss of generality, that $f(p)=0$ and $z_0(p)=0$, i.e., $p\in V(f, z_0)$.

Consider the cotangent bundle of $\U$, 
$$\pi: T^*\U\cong \U\times \C^{n+1}\rightarrow\U.$$
 For all $x\in\U$, we use $(d_xz_0, \dots, d_xz_n)$ as an ordered basis for the fiber $(T^*\U)_x:=\pi^{-1}(x)$.
 
The (closure of the) relative conormal space of $f$ in $\U$, $\relcon\subseteq T^*\U$, is a well-known object in the study of the singularities of the hypersurface $V(f)$; see, for instance,  \cite{teissiervp2}. The relative conormal space is given by
 $$
\relcon:=\overline{\big\{(x, \eta)\in T^*\U\ | \ \eta(\operatorname{ker}(d_xf))\equiv0\big\}}.
 $$
Note that we have not done what is usually done, in that we have not explicitly removed the critical locus, $\Sigma f$, of $f$ before closing; this does not affect the closure. We remark that each fiber of $\relcon$ over $\U$ is $\C$-conic, i.e., closed under scalar multiplication, but need not be closed under addition over a point in the critical locus of $f$.

\bigskip

For all $x\in V(z_0)$, there is a canonical map 
$$\hat r_x: (T^*\U)_x\rightarrow (T^*(V(z_0)))_x
$$
given by 
$$
\hat r_x(\eta):=\eta_{|_{T_xV(z_0)}},
$$
that is, 
$$\hat r_x(a_0d_xz_0+\cdots+a_nd_xz_n) = a_1d_xz_1+\cdots+a_nd_xz_n.
$$

\bigskip

We prove just one theorem in this paper:

\begin{thm} For all $x\in V(z_0)$, if $d_xz_0\not\in \big(\relcon\big)_x$, then $\hat r_x$ induces a surjection
$$
r_x: \big(\relcon\big)_x\rightarrow \big(\overline{T^*_{f_{|_{V(z_0)}}}V(z_0)}\big)_x,
$$
such that $r_x^{-1}(0)=0$.
\end{thm}

The proofs that $r_x^{-1}(0)=0$ and that $r_x$ is a surjection are easy, but that everything in $\big(\relcon\big)_x$ maps by $r_x$ into $\big(\overline{T^*_{f_{|_{V(z_0)}}}V(z_0)}\big)_x$ is difficult, and our argument heavily uses the derived category and the microsupport.  Our attempts to find a more direct proof, or to find this result in the existing literature, have failed.

\medskip

And so we must begin by looking at derived category definitions and results.

\section{Derived Category Results} 

We must begin by recalling a large number of definitions and notations.

\medskip

Consider a (reduced) complex analytic subspace $X$ of some open subset $\W$ of some affine space $\C^N$.

We will look at objects in the derived category $D^b_c(X)$ of bounded, constructible complexes of sheaves of $\Z$-modules on various spaces. References for the notation and results that we will use are \cite{kashsch}, \cite{dimcasheaves}, \cite{singenrich}, \cite{micromorse}, and \cite{ipaaffine}.

\smallskip

For a complex submanifold $M\subseteq \W$, we let $T^*_M\W$ denote that conormal space of $M$ in $\W$; this is the subspace of the cotangent space $T^*\W$ given by
$$
T^*_M\W:=\big\{(p, \eta)\in T^*\W \ | \ \eta(T_pM)\equiv 0\big\}.
$$
We will usually be interested in the closure $\overline{T^*_M\W}$ in $T^*\W$.

\smallskip

Suppose that $\Adot\in D^b_c(X)$ and that $\strat$ is a Whitney stratification (with connected strata) of $X$ with respect to which $\Adot$ is constructible.  Then, as described by Goresky and MacPherson \cite{stratmorse}, to each stratum $S$ in $\strat$, there are an associated normal slice $\N_S$ and complex link $\mathbb L_S$. The isomorphism-types of the hypercohomology modules $\hyp^*(\N_S, \mathbb L_S; \Adot)$ are independent of the choices made in defining the normal slice and complex link; these are the {\it Morse modules} of $S$, with respect to $\Adot$. We let $m_S^k(\Adot):=\hyp^{k-\dim S}(\N_S, \cL_S; \Adot)$. 

The union of the closures of conormal spaces to strata with non-zero Morse modules is the {\it microsupport}, $\ms(\Adot)$, of $\Adot$, as defined by Kashiwara and Schapira in \cite{kashsch}, i.e.,   
$$\ms(\Adot):= \bigcup_{m^*_S(\Adot)\neq 0}\overline{T^*_S\W}.
$$
(For this characterization of the microsupport, see Theorem 4.13 of \cite{micromorse}.) The microsupport is independent of the choice of Whitney stratification.


\vskip 0.2in

Now we return to the notation from the introduction:  $\U$ is a connected, open subset of $\C^{n+1}$, $(z_0, \dots, z_n)$ are coordinates on $\U$, $f:\U\rightarrow\C$ is a non-constant complex analytic function, and $p\in V(f, z_0)$. We now fix $\Pdot$ as being the shifted constant sheaf $\Z^\bullet_\U[n+1]$.

We will use the (shifted) {\it nearby and vanishing cycles functors}, $\psi_f[-1]$ and $\phi_f[-1]$, respectively, from $D^b_c(\U)$ to $D^b_C(V(f))$. Let $F_{f, p}$ denote the Milnor fiber of $f$ at $p$, and let $f_0:=f_{|_{V(z_0)}}$. Let $\hat z_0:={z_0}_{|_{V(f)}}$.

We define the inclusions $$m: V(z_0)\hookrightarrow\U, \hskip 0.2in l: \U\backslash V(z_0)\hookrightarrow\U, \ \textnormal{ and }\ \hat m: V(f,z_0)\hookrightarrow V(f).$$

Then, $H^k(\psi_f[-1]\Pdot)_\p\cong H^{k+n}(F_{f, p}; \Z)$ and $H^k(\phi_f[-1]\Pdot)_\p\cong \tilde H^{k+n}(F_{f, p}; \Z)$, where $\widetilde H$ denotes reduced cohomology. Furthermore,
$$
H^k(\psi_f[-1]l_!l^!\Pdot)_\p\cong H^{k+n}(F_{f, p}, F_{f_0, p}; \Z)
$$
and
$$
\hat m^*\psi_f[-1]m_*m^*[-1]\Pdot \cong \psi_{f_0}[-1]\Z^\bullet_{V(z_0)}[n].
$$

\bigskip

It is a fundamental result of Brian\c con, Maisonobe, and Merle that:

\begin{thm}\label{thm:bmm}\textnormal{(\cite{bmm}, 3.4.2)} 
$$
\ms(\psi_f[-1]\Pdot)=\relcon\cap\big(V(f)\times \C^{n+1}\big)=\relcon\cap \pi^{-1}f^{-1}(0).
$$

More generally, if $$\ms(\Adot)=\bigcup_{S\in \strat}\overline{T^*_S\U},$$ then
$$
\ms(\psi_f[-1]\Adot) = \Big(\bigcup_{S\not\subseteq V(f)}\overline{T^*_{f_{|_{S}}}}\U\Big)\cap \big(V(f)\times \C^{n+1}\big).
$$
\end{thm}
(Actually, the statement above does not use the full result of  \cite{bmm}, 3.4.2.)

\bigskip

Now, we are finally in a position to state and prove some results.

\section{The Main Result}

\begin{prop}\label{prop:onlyprop} Suppose that $p\in V(f, z_0)$ and that $d_pz_0\not\in \big(\relcon\big)_p$. Then, near $p$, we have an isomorphism
$$
 \psi_{f_0}[-1]\Z^\bullet_{V(z_0)}[n]\cong \psi_{\hat z_0}[-1]\psi_f[-1]\Pdot.
$$
\end{prop}
\begin{proof} We will use Theorem 3.1 of \cite{ipaaffine}. Since the set of $x$ such that $d_xz_0\in \big(\relcon\big)_x$ is closed, the fact that $d_pz_0\not\in \big(\relcon\big)_p$ implies that, for all $x$ near $p$, $d_xz_0\not\in \big(\relcon\big)_x$. This implies that the relative polar set $\Gamma_{f, z_0}$ is empty near $p$, which implies, by Theorem 3.1 of \cite{ipaaffine}, that, for all $x\in V(f, z_0)$ near $p$,
$$
H^*(\psi_f[-1]l_!l^!\Pdot)_x\cong H^*(F_{f, x}, F_{f_0, x}; \Z) =0,
$$
that is, near $p$,
$$
\hat m^*\psi_f[-1]l_!l^!\Pdot=0.
$$

Using this last equality, and applying the functor $\hat m^*[-1]\psi_f[-1]$ to the canonical distinguished triangle
$$
l_!l^!\Pdot\rightarrow\Pdot\rightarrow m_*m^*\Pdot\arrow{[1]}l_!l^!\Pdot,
$$
we conclude that, near $p$,
\begin{equation}
\hat m^*[-1]\psi_f[-1]\Pdot\cong \hat m^*\psi_f[-1]m_*m^*[-1]\Pdot. \tag{$\dagger$}
\end{equation}

There is also the canonical distinguished triangle relating the nearby and vanishing cycles along $\hat z_0$:
$$
\hat m^*[-1]\psi_f[-1]\Pdot\rightarrow \psi_{\hat z_0}[-1]\psi_f[-1]\Pdot\rightarrow \phi_{\hat z_0}[-1]\psi_f[-1]\Pdot\arrow{[1]}\hat m^*[-1]\psi_f[-1]\Pdot.
$$
Now Theorem 3.1 of \cite{ipaaffine} tells us that, near $p$, we also know that $\phi_{\hat z_0}[-1]\psi_f[-1]\Pdot=0$. Thus, near $p$,
\begin{equation}
\hat m^*[-1]\psi_f[-1]\Pdot\rightarrow \psi_{\hat z_0}[-1]\psi_f[-1]\Pdot.\tag{$\ddagger$}
\end{equation}

\medskip

Combining ($\dagger$) and ($\ddagger$) yields the desired result.
\end{proof}

\bigskip

We can now prove the theorem that we stated in the introduction.

\begin{thm}\label{thm:onlythm} For all $x\in V(z_0)$, if $d_xz_0\not\in \big(\relcon\big)_x$, then $\hat r_x$ induces a surjection
$$
r_x: \big(\relcon\big)_x\rightarrow \big(\overline{T^*_{f_{|_{V(z_0)}}}V(z_0)}\big)_x,
$$
such that $r_x^{-1}(0)=0$.
\end{thm}
\begin{proof} The final statement is trivial to prove. Suppose that $d_xz_0\not\in \big(\relcon\big)_x$; as $\big(\relcon\big)_x$ is $\C$-conic, it follows that, if $ad_xz_0\in \big(\relcon\big)_x$, then $a=0$. The statement now follows from 
$$r^{-1}_x(0)=\big(\relcon\big)_x\cap \hat r_x^{-1}(0)= \{ad_xz_0\in \big(\relcon\big)_x \ | \ a\in\C\}.
$$

\bigskip

Now, suppose that $x\in V(z_0)$ and $d_xz_0\not\in \big(\relcon\big)_x$. If $c$ is a constant, replacing $f$ by $f-c$ does not affect $\relcon$, and so we may assume that $f(x)=0$.

\bigskip

We wish to look at the equality of the microsupports of the two isomorphic complexes in  \propref{prop:onlyprop}.  While both complexes in the isomorphism in \propref{prop:onlyprop} are complexes on $V(f, z_0)$, the results that we shall use will tell us one microsupport in $T^*V(z_0)$ and the other in $T^*\U$; we shall write $\ms_{{}_{V(z_0)}}$ and $\ms_{{}_\U}$, respectively.

If $Y$ is an analytic subset of $V(z_0)\subset\U$ and $\Adot\in D^b_c(Y)$, we may consider the microsupport $\ms_{{}_{V(z_0)}}(\Adot)$ in $T^*V(z_0)$ or $\ms_{{}_{\U}}(\Adot)$ in $T^*\U$. The relationship between these is trivial:
$$
\ms_{{}_{\U}}(\Adot)=\ms_{{}_{V(z_0)}}(\Adot)+<dz_0>:=\big\{(x, \eta+ad_xz_0) \ | \ (x, \eta)\in \ms_{{}_{V(z_0)}}(\Adot), a\in\C\big\}.
$$

\bigskip

We shall always work near $x$, and finally look at precisely the fiber over $x$.

\bigskip

By \thmref{thm:bmm}, 
$$
\ms_{{}_{V(z_0)}}( \psi_{f_0}[-1]\Z^\bullet_{V(z_0)}[n]) = \overline{T^*_{f_{|_{V(z_0)}}}V(z_0)}\cap (V(f, z_0)\times \C^n),
$$
and so
\begin{equation}
\ms_{{}_{\U}}( \psi_{f_0}[-1]\Z^\bullet_{V(z_0)}[n]) =\Big( \overline{T^*_{f_{|_{V(z_0)}}}V(z_0)}\cap (V(f, z_0)\times \C^n)\Big)+<dz_0>.\tag{$*$}
\end{equation}

We must do a little more work to find $\ms(\psi_{\hat z_0}[-1]\psi_f[-1]\Pdot)$.

Again, by  \thmref{thm:bmm}, 
$$
\ms_{{}_\U}(\psi_f[-1]\Pdot) =\bigcup_{R\in\mathfrak R}\overline{T^*_R\,\U} = \relcon\cap\big(V(f)\times \C^{n+1}\big),
$$
where $\mathfrak R$ is some subset of a Whitney stratification with respect to which $\psi_f[-1]\Pdot$ is constructible. Note that since are working at points $x$ such that $d_xz_0\not\in \big(\relcon\big)_x$, if $x\in V(f)$, then, for each $R\in\mathfrak R$, $d_xz_0\not\in \overline{T^*_R\,\U}$.

Now, using the more general statement in \thmref{thm:bmm}, we have that
\begin{equation}
\ms_{{}_\U}(\psi_{\hat z_0}[-1]\psi_f[-1]\Pdot) = \Big(\bigcup_{R\not\subseteq V(z_0)}\overline{T^*_{{z_0}_{|_{R}}}\U}\Big)\cap (V(z_0)\times\C^{n+1}). \tag{$**$}
\end{equation}

By \propref{prop:onlyprop}, 
$$
\ms_{{}_\U}(\psi_{\hat z_0}[-1]\psi_f[-1]\Pdot)=\ms_{{}_{\U}}( \psi_{f_0}[-1]\Z^\bullet_{V(z_0)}[n]),
$$
and, thus, at $x$, we have
$$
\big(\overline{T^*_{f_{|_{V(z_0)}}}V(z_0)}\big)_x+<d_xz_0>= \Big(\bigcup_{R\not\subseteq V(z_0)}\overline{T^*_{{z_0}_{|_{R}}}\U}\Big)_x.
$$

By Lemma 5.8 of \cite{enrichpolar}, since for each $R\in\mathfrak R$, $d_xz_0\not\in \overline{T^*_R\,\U}$,
$$
\Big(\bigcup_{R\not\subseteq V(z_0)}\overline{T^*_{{z_0}_{|_{R}}}\U}\Big)_x \ =  \ \Big(\bigcup_{R\in\mathfrak R}\overline{T^*_R\,\U}\Big)_x+<d_xz_0> \  = \ \big(\relcon\big)_x+<d_xz_0>.
$$

Therefore, we conclude that
$$
\big(\overline{T^*_{f_{|_{V(z_0)}}}V(z_0)}\big)_x+<d_xz_0> \ = \ \big(\relcon\big)_x+<d_xz_0>,
$$
and, hence,
$$
r_x\big(\big(\relcon\big)_x\big) = r_x\big(\big(\relcon\big)_x+<d_xz_0>\big)=
$$

$$
r_x\big(\big(\overline{T^*_{f_{|_{V(z_0)}}}V(z_0)}\big)_x+<d_xz_0>\big)= \big(\overline{T^*_{f_{|_{V(z_0)}}}V(z_0)}\big).
$$

\end{proof}

\section{Concluding Remarks}

As $\relcon\subseteq \U\times\C^{n+1}$ is $\C$-conic in the cotangent coordinates, it projectivizes to $\Proj(\relcon)\subseteq\U\times \Proj^n$; in fact, for many authors, this projective object is what they mean by {\bf the} relative conormal space.

 As $r_x$ preserves scalar multiplication and as $r_x^{-1}(0)=0$,  \thmref{thm:onlythm} immediately implies a projective version of itself. We denote the projective class of $d_xz_0$ by $[d_xz_0]$.

\begin{thm}For all $x\in V(z_0)$, if $[d_xz_0]\not\in \Proj\big(\relcon\big)_x$, then $\hat r_x$ induces a surjection
$$
\check r_x: \Proj\big(\relcon\big)_x\rightarrow \Proj\big(\overline{T^*_{f_{|_{V(z_0)}}}V(z_0)}\big)_x.
$$
\end{thm}

\bigskip

Our primary interest in  \thmref{thm:onlythm} relates to the L\^e cycles and L\^e numbers of $f$; see, for instance, \cite{lecycles}. In many of our past works, the genericity condition that we required to guarantee the existence of the L\^e cycles and numbers was that the coordinates $(z_0, \dots, z_n)$ be {\it prepolar}. This is an inductive requirement on how the hyperplanes slices $V(z_i)$ intersect the strata of an $a_f$ stratification of $V(f)$. Ideally, we would like to eliminate the need to first produce an $a_f$ stratification. It turns out that \thmref{thm:onlythm}, or its projective version, is precisely the lemma that we need.

\newpage
\bibliographystyle{plainurl}

\end{document}